\documentclass[11pt,a4paper]{article}

\usepackage{tikz}
\usepackage{subfigure}
\usepackage[english]{babel}
\usepackage{graphicx}
\usepackage{extarrows}

\usepackage[center]{caption2}
\usepackage{amsfonts,amssymb,amsmath,latexsym,amsthm}
\usepackage{multirow}
\DeclareSymbolFont{bbold}{U}{bbold}{m}{n}
\DeclareSymbolFontAlphabet{\mathbbold}{bbold}

\usepackage{ebezier,eepic}
\usepackage{color}
\usepackage{multirow}

\usepackage{epsf,epsfig,amsfonts,amsgen,amsmath,amstext,amsbsy,amsopn,amsthm,amsfonts,amssymb,amscd
}
\usepackage{ebezier,eepic}
\usepackage{color}
\usepackage{multirow}
\def\ex{\mathrm{ex}}

\def\F{\mathcal{F}}

\def\J{\mathcal{J}}

\def\cF{\F}

\def\cJ{\J}
\def\cC{\mathcal{C}}
\def\cD{\mathcal{D}}
\def\cP{\mathcal{P}}
\def\cQ{\mathcal{Q}}
\def\cS{\mathcal{S}}

\def\cT{\mathcal{T}}
\def\cH{\mathcal{H}}

\newtheorem{thm}{Theorem}[section]
\newtheorem{cor}[thm]{Corollary}
\newtheorem{lem}[thm]{Lemma}
\newtheorem{prop}[thm]{Proposition}
\newtheorem{conj}[thm]{Conjecture}

\newtheorem{Definition}{Definition}

\begin{document}

\title{Many Tur\'an exponents via subdivisions}
\author{
	Tao Jiang\thanks{Department of Mathematics, Miami University, Oxford,
		OH 45056, USA. E-mail: jiangt@miamioh.edu. Research supported in part by NSF grant DMS-1855542.}
	\quad \quad Yu Qiu \thanks{
		School of Mathematical Sciences,
		University of Science and Technology of China, Hefei, 230026,
		P.R. China. Email: yuqiu@mail.ustc.edu.cn. Research supported by China Scholarship Council.
} }

\date{August 6, 2019}
\maketitle
\begin{abstract}
Given a graph $H$ and a positive integer $n$,  the {\it Tur\'an number} $\ex(n,H)$ is the maximum number of edges in an $n$-vertex graph that does not contain $H$ as a subgraph.
	A real number $r\in(1,2)$ is called a {\it Tur\'an exponent} if there exists a bipartite graph $H$ such that $\ex(n,H)=\Theta(n^r)$. A long-standing conjecture of Erd\H{o}s and Simonovits states that $1+\frac{p}{q}$ is a Tur\'an exponent for all positive integers $p$ and $q$ with $q> p$. 
	
	  In this paper, we build on recent developments on the conjecture to establish a large family of new Tur\'an exponents. In particular, it follows from our main result that $1+\frac{p}{q}$ is a Tur\'an exponent for all positive integers $p$ and $q$ with $q> p^2$.
\end{abstract}

\section{Introduction}
\subsection{Rational exponent conjecture}
Given a family $\cH$ of graphs, the Tur\'an number $\ex(n,\cH)$ is the largest number of edges in an $n$-vertex graph that does not contain any member of $\cH$ as a subgraph. When $\cH$ consists of one single graph $H$, we write $\ex(n,H)$ for $\ex(n,\{H\})$.

Determining Tur\'an numbers for various graphs is one of the central problems in extremal graph theory. The celebrated Erd\H{o}s-Stone-Simonovits theorem states that for any non-bipartite graph $H$, $\ex(n,H)=(1-\frac{1}{\chi(H)-1})\binom{n}{2}+o(n^2)$, where $\chi(H)$ is the chromatic number of $H$. 
For bipartite graphs $H$, it follows from the K\H{o}vari-S\'os-Tur\'an theorem that $\ex(n,H)=O(n^{2-\alpha})$, where $\alpha=\alpha_H>0$ is a constant. However, finding good estimates
on $\ex(n,H)$ for bipartite graphs $H$ is difficult. Until recently,
the order of magnitude of $\ex(n,H)$ is known only for very few bipartite graphs $H$. Following \cite{KKL},
we say that a real number $r\in(1,2)$  is 
{\it realizable} (by $H$) if there exists a bipartite graph $H$ such that $\ex(n,H)=\Theta(n^r)$. 
If $r$ is realizable then we also call it a {\it Tur\'an exponent}.
 A well-known conjecture of Erd\H{o}s and Simonovits, known as the {\it rational exponent conjecture}, asserts that every rational number $r\in(1,2)$ is a Tur\'an exponent. 

\begin{conj}{\rm \cite{Erdos}}  \label{conj:exponent}
	For all positive integers $q> p$, $1+\frac{p}{q}$ is a Tur\'an exponent.
\end{conj}

Until recently, the only rationals  in $(1,2)$ for which the conjecture was known to be true were rationals of the form $1+\frac{1}{q}$ and $2-\frac{1}{q}$ for positive integers $q\geq 2$, realized by so-called theta graphs and complete bipartite graphs, respectively. In a recent breakthrough work, Bukh and Conlon \cite{BC} showed that for any rational number $r\in(1,2)$, there exists a finite family $\cH_r$ of graphs such that $\ex(n,\cH_r)=\Theta(n^r)$. Bukh and Conlon's work  has, to a large extent, rejuvenated people's interest on Conjecture \ref{conj:exponent}. In the last year or so, several new infinite sequences of new Tur\'an exponents have been obtained by various groups. First, Jiang, Ma, and Yepremyan \cite{JMY} showed that $2-\frac{2}{2m+1}$ is realizable by generalized cubes and that $\frac{7}{5}$ is realizable by the so-called $3$-comb-pasting graph. A few months later, Kang, Kim, and Liu \cite{KKL} showed that for all positive integers $p<q$, where $q\equiv \pm1 \pmod{p}$, $2-\frac{p}{q}$ is realizable. More specifically,  rationals of the form $2-\frac{t}{st-1}$, where $s,t\geq 2$, are realized by the so-called {\it blowups} of certain height $2$ trees. (We will define blowups precisely in subsection 1.2.)
Rationals of the form $2-\frac{t}{st+1}$ are realized by graphs obtained from theta graphs via some iterative operations. 
More recently, some new sequences of Tur\'an exponents were obtained along the study of Tur\'an numbers of subdivisions.
For any integers $s,t\geq 1, k\geq 2$, let $K_{s,t}^k$ denote the graph obtained from the complete bipartite graph $K_{s,t}$ by
subdividing each of its edge $k-1$ times. Let $L_{s,t}(k)$ by obtained from $K_{s,t}^k$ by adding an extra vertex joined
to all vertices in the part of $K_{s,t}$ of size $t$.
Confirming a conjecture of Kang, Kim, and Liu \cite{KKL}, Conlon, Janzer, and Lee \cite{CJL} showed that 
there exists $t_0$ such that for all  integers $s,k\geq 1, t\geq t_0$, $\ex(n,L_{s,t}(k))=\Theta(n^{1+\frac{s}{sk+1}})$, and thus establishing 
$1+\frac{s}{sk+1}$ as Tur\'an exponents. Subsequently, in verifying a conjecture of Conlon, Janzer, and Lee \cite{CJL}, Janzer \cite{Janzer2} proved that there exists a $t_0$ such that for all integers $s,k\geq 2, t\geq t_0$, $\ex(n,K_{s,t}^k)=\Theta(n^{1+\frac{s-1}{sk}})$, thus establishing $1+\frac{s-1}{sk}$ as Tur\'an exponents. Earlier, 
 Conlon, Janzer, Lee \cite{CJL} had proven the conjecture for $k=2$, while Jiang and Qiu \cite{JQ} proved the conjecture for $k=3,4$.

\subsection{Our results}

In this paper, we build on the recent work on subdivisions to establish the following large three-parameter family of Tur\'an exponents, which include all the ones obtained by Conlon, Janzer, and Lee \cite{CJL} and by Janzer  \cite{Janzer2}.
\begin{thm}\label{thm:main}
	For any positive integers $p,k,b$ with $k\ge b$, $1+\frac{p}{kp+b}$ is a Tur\'an exponent.
\end{thm}

As an immediate corollary, we get the following easily stated result.
\begin{thm}
	For any positive integers $p$ and $q$ with $q> p^2$, $1+\frac{p}{q}$ is a Tur\'an exponent.
\end{thm}

Using a reduction lemma of Kang, Kim, and Liu \cite{KKL}, Theorem \ref{thm:main} also yields
\begin{cor} \label{cor:dense1}
	For any integers $b,p,s\ge1$ and $k\ge0$, if $k\ge b-1$, then $2-\frac{kp+b}{s(kp+b)+p}$ is a Tur\'an exponent.
\end{cor}

Corollary \ref{cor:dense1} implies the following.

\begin{cor} \label{cor:dense2}
For any positive integers $p,q$ with $q>p$, if $(q \mbox{ mod } p)\le\sqrt{p}$, then $2-\frac{p}{q}$ is a Tur\'an exponent.
\end{cor}

Theorem \ref{thm:main} follows from a theorem (Theorem \ref{thm:subdivision-bound}) that we prove on the Tur\'an number of subdivisions of $K_{s,t}$ where
different edges of $K_{s,t}$ may be subdivided different number of times. The theorem is interesting on its own and
partially answers a conjecture of Janzer (Conjecture \ref{conj:Janzer}), which we will describe in the next subsection.

\subsection{The Bukh-Conlon Conjecture and Janzer's conjecture}

At the core of the work of Bukh and Conlon \cite{BC} is the study of so-called blowups of balanced rooted trees, defined as follows (also see \cite{BC}).

\begin{Definition}
	A rooted tree $(T,R)$ consists of a tree $T$ together with an independent set $R\subseteq V(T)$, which we refer to as the roots. 	When the choice of $R$ is clear, we will simply write $T$ for $(T,R)$.
\end{Definition}

\begin{Definition}
	Given a rooted tree $(T,R)$ and a non-empty subset $S\subseteq V(T)\setminus R$, let $\rho_{T}(S)=\frac{e(S)}{|S|}$, where $e(S)$ is the number of edges in $T$ that have at least one end in $S$. Let $\rho_T=\rho_T(V(T)\setminus R)$ and call it the density of $T$. We say $(T,R)$ is balanced if $\rho_{T}(S)\ge \rho(T)$ for any non-empty subset $S\subseteq V(T)\setminus R$.
\end{Definition}
 
\begin{Definition}
	The $t$-blowup of a rooted tree $(T,R)$, denoted by $t*T_{R}$, is the union of $t$ labeled copies of $T$ which agree on $R$ but are pairwise vertex-disjoint outside $R$. If  the choice of $R$ is clear, then we write $t*T$ for $t*T_{R}$.
\end{Definition}

The key result of Bukh and Conlon \cite{BC} is the following lower bound theorem, established using an innovative random algebraic approach. Interested readers can find the full statement in \cite{BC}.
\begin{thm}{\rm \cite{BC}} \label{BC-theorem}
	Suppose that $(T,R)$ is a balanced rooted tree with density $\rho$. Then there exists an integer $t_0\ge2$ such that
	for all integers $t\geq t_0$ we have $\ex(n,t*T_R) = \Omega(n^{2-\frac1\rho})$. 
\end{thm}

Bukh and Conlon further made the following conjecture on a matching upper bound.

\begin{conj} {\rm \cite{BC}} \label{conj:BC}
	Suppose that $(T,R)$ is a balanced rooted tree with density $\rho$. Then for all positive integers $t$  we have $\ex(n,t*T_R) = O(n^{2-\frac1\rho})$. 
\end{conj}

Besides being interesting on its own, a significance of Conjecture \ref{conj:BC} is that it implies the rational exponent conjecture. Indeed, for each rational $r\in (1,2)$, Bukh and Conlon were able to construct a balanced rooted tree $(T,R)$ with density $\rho=\frac{1}{2-r}$. Hence Theorem \ref{BC-theorem} and Conjecture \ref{conj:BC} together would give $\ex(n,t*T_R)=\Theta(n^r)$ for some sufficiently large positive integer $t$. A careful reader will note that Bukh and Conlon's conjecture is in fact much stronger than the rational exponent conjecture. Indeed, to prove the rational exponent conjecture, it suffices to search, for each $r\in(1,2)$, a balanced rooted tree $(T,R)$ with density $\rho=\frac{1}{2-r}$ for which the Bukh-Conlon conjecture holds.  This suggests that one way to make further progress on the rational exponent conjecture is to find suitable balanced rooted trees to explore Conjecture \ref{conj:BC} with. One family of trees whose exploration has brought some success  are the so-called spiders.

\begin{Definition}
	Let $s\geq 2$ be an integer. An $s$-legged spider $S$ with center $u$ is a tree consisting of $s$ paths (called the legs of $S$) that share one common end $u$ but are vertex-disjoint outside $u$. Moreover, we say $S$ has length vector $(j_1,\ldots,j_s)$ and leaf vector $(x_1,\ldots,x_s)$ if for every $1\le i\le s$, its $i$-th leg has length $j_i$ and has ends $u$ and $x_i$. 
\end{Definition}

For spiders with roots being all of its leaves, checking balancedness is simple.
\begin{prop} Let $s, k$ be integers where $s\geq 2, k\geq 1$. Let $S$ be an $s$-legged spider and $R$
the set of its leaves, Suppose the longest leg of $S$ has length $k$. Then
$(S,R)$ is a balanced rooted tree if and only if $e(S)\geq (s-1)k$.
\end{prop}
When $S$ is an $s$-legged spider with length vector $(k,\ldots,k)$ and $R$ is the set of its leaves,
$t*S_R$ is the subdivision $K_{s,t}^k$ of $K_{s,t}$, considered by Janzer \cite{Janzer2}. 
When $S$ is an $(s+1)$-legged spider with length vector $(1,k,\ldots, k)$ 
and $R$ is the set of its leaves, $t*S_R$ is the graph $L_{s,t}(k)$, considered by Conlon, Janzer, and Lee \cite{CJL}.
Motivated by the earlier mentioned results on $\ex(n,L_{s,t}(k))$  and $\ex(n,K_{s,t}^k)$,
Janzer \cite{Janzer2} made the following conjecture.

\begin{conj} [{\rm \cite{Janzer2}}] \label{conj:Janzer}
Let $s\geq 2, k,b,t\geq 1$ be integers. Let $S$ be an $s$-legged spider where the longest leg has length $k$. Suppose that 
$e(S)=(s-1)k+b$, where $0\leq b \leq k$. Then $\ex(n, t*S)= O(n^{1+\frac{s-1}{(s-1)k+b}})$.
\end{conj}
Even though Janzer's conjecture is a special case of the Bukh-Conlon conjecture, it is also interesting on its own due to
its connection to the study of subdivisions.
Let $S$ be as specified in Conjecture \ref{conj:Janzer}.
It follows from Theorem \ref{BC-theorem} that there exists a $t_0$ such that for all $t\geq t_0$, 
$\ex(n,t*S)=\Omega(n^{1+\frac{s-1}{(s-1)k+b}})$. Hence, if Conjecture \ref{conj:Janzer} is true, it will establish all rationals of the
form $1+\frac{p}{pk+b}$ as Tur\'an exponents, where $p,k$ are positive integers and $b$ is an integer with $0\leq b \leq k$.
Here, we settle an important case of Conjecture \ref{conj:Janzer} that allows us to obtain all the Tur\'an exponents
that Conjecture \ref{conj:Janzer} would give.

\begin{Definition}
	For positive integers $k,b$ and $s$, let $S^s_{b,k}$ denote the $s$-legged spider with length vector $(b,k,\ldots,k)$.
\end{Definition}

Using this notation, we have $K_{s,t}^k=t*S^s_{k,k}$ and $L_{s,t}(k)=t*S^{s+1}_{1,k}$.
In this paper, we will prove the following common generalization of the
 result of Conlon, Janzer, and Lee on $\ex(n, L_{s,t}(k))$ and the result of Janzer on $\ex(n,K_{s,t}^k)$,
 from which our main theorem, Theorem \ref{thm:main}, follows.

\begin{thm}\label{thm:subdivision-bound}
	For any $s,t\ge 2$ and $k\ge b\ge 1$, $\ex(n,t*S^s_{b,k})= O(n^{1+\frac{s-1}{(s-1)k+b}})$.
\end{thm}

As in \cite{CL, CJL, JQ, Janzer2}, we will use the following variant of the regularization lemma of Erd\H{o}s and Simonovits
\cite{cube},  as given in \cite{JS}. Given a positive constant $K$, a graph $G$ is {\it $K$-almost-regular} if
$\Delta(G)\leq K\delta(G)$.

\begin{lem} {\rm \cite{JS}} \label{lem: almost-regular}
	Let $0<\epsilon<1$ and $c\ge1$. There exists $n_0=n_0(\epsilon)>0$ such that the following holds for all $n\ge n_0$. If $G$ is a graph on $n$ vertices with $e(G)\ge cn^{1+\epsilon}$, then $G$ contains a $K$-almost-regular subgraph $G'$ on $m\ge n^{\frac{\epsilon-\epsilon^2}{2+2\epsilon}}$ vertices such that $e(G')\ge\frac{2c}5m^{1+\epsilon}$ and $K=\lceil20\cdot2^{\frac1{\epsilon^2}+1}\rceil.$
\end{lem}

By Lemma \ref{lem: almost-regular}, in order to prove Theorem \ref{thm:subdivision-bound}, it suffices to prove the following.

\begin{thm}\label{thm:regular}
	Let $s,t\ge 2$ and $k\ge b\ge 1$. Let $K=K(s,b,k)$ be obtained by Lemma \ref{lem: almost-regular} with $\epsilon:=\frac{s-1}{(s-1)k+b}$.  There exist positive constants $n_0$ and $C$ depending only on $s,t,b,k$ such that for all integers $n\geq n_0$ if $G$
is an $n$-vertex  $t*S^s_{b,k}$-free $K$-almost-regular graph then $\delta(G)< C n^\frac{s-1}{(s-1)k+b}$.
\end{thm}

The rest of the paper is organized as follows. In Section 2, we introduce some  notation and preliminary lemmas.
In Section 3, we prove  Theorem \ref{thm:regular}, from which Theorems \ref{thm:subdivision-bound} and \ref{thm:main}  follow. In Section 4, we give a sketch of proofs of Corollaries \ref{cor:dense1} and \ref{cor:dense2} and some concluding remarks.


\section{Notation and preliminaries}

Given a positive integer $m$, let $[m]=\{1,\dots, m\}$.
Given a graph $G$ and a vertex $w$, for  each $i\ge1$ let $\Gamma_i(w)$ be the set of vertices $z$ such that there exists a path in $G$ of length $i$ with ends $w$ and $z$. When $i=1$, we often write $N_G(w)$ for $\Gamma_1(w)$. Let $e(G)$ be the number of edges in $G$. We use standard asymptotic notations, i.e., given two positive functions $f(n)$ and $g(n)$, by  $f=o_n(g),f=\omega_n(g),f=\Omega_n(g),f=O_n(g),f=\Theta_n(g)$,  we respectively mean $\lim_{n\rightarrow\infty}f/g=0,\liminf_{n\rightarrow\infty}f/g=\infty,\liminf_{n\rightarrow\infty}f/g>0,\limsup_{n\rightarrow\infty}f/g<\infty,0<\liminf_{n\rightarrow\infty}f/g\le\limsup_{n\rightarrow\infty}f/g<\infty$.  Whenever the context is clear, we drop the subscript $n$.
If $G$ is a graph and $S$ is a set of vertices in it, then we define
\[N^*_G(S)=\bigcap_{x\in S} N_G(x),\]
and call it the {\it common neighborhood} of $S$ in $G$.

For the rest of the paper, we fix integers $s,t\ge 2$ and $k\ge b\ge 1$, and let $K=K(s,b,k)$ be obtained by Lemma \ref{lem: almost-regular} with $\epsilon:=\frac{s-1}{(s-1)k+b}$. 

Below are some key concepts introduced in \cite{CJL}, which we adapt for our setting.

\begin{Definition}\label{def:f-definition} 
	Let $L$ be a positive integer, we define $f(1,L)=L$ and for $j\ge2$,
	\begin{equation*}
	f(j,L):=10j^4 [2K^jL\cdot f(j-1,L)^2]^{s+3}.
	\end{equation*}
\end{Definition}

We will need the following property of the function in various places of the paper.

\begin{prop}\label{prop}
	For every integer $j\geq 2$, $\frac{f(j,L)}{j^2f(j-1,L)^2}\ge\max\{2L^2,f(j-1,L)\}$ holds.
\end{prop}		
	
The next two definitions are crucial to our overall arguments.

\begin{Definition} \label{def:admissible-paths}
	We recursively define $j$-admissible, $j$-light paths, and $j$-heavy paths in a graph $G$. Any edge is both 1-admissible and 1-light. For $j\ge2$, a path $P$ is $j$-admissible if it has length $j$ and for each $1\leq \ell<j$
	every subpath of length $\ell$ in $P$ is $\ell$-light. 
	
Among $j$-admissible paths $P$ with ends $x$ and $y$, we further say that $P$ is $j$-light if the number of $j$-admissible paths with ends $x$ and $y$ in $G$ is less than $f(j,L)$ and that $P$ is $j$-heavy otherwise.
\end{Definition}
Since the length of a path $P$ is fixed, we often drop the prefix $j$ and $\ell$ in the definitions above.
Note that $j$-admissible and $j$-light paths are defined for all $j\geq 1$ while $j$-heavy paths are defined only for $j\geq 2$.
In \cite{JQ}, the concepts of admissible, light, and heavy paths were extended for spiders. Here, we adapt the definitions from 
\cite{JQ} further.

\begin{Definition} \label{def:admissible-spiders}
	We recursively define $s$-legged admissible, light, and heavy spiders in a graph $G$. Any spider of height 1 is both admissible and light. Let $S$ be an $s$-legged spider with leaf vector $(x_1,\ldots, x_s)$ and length vector $(j_1,\dots, j_s)\neq (1,1,\dots,1)$. We say that $S$ is admissible if every leg of it is a light path as defined in Definition \ref{def:admissible-paths} and 
	every $s$-legged proper sub-spider of $S$ is light. Suppose $S$ is admissible. Then we 
	further say that it is light if
	the number of admissible spiders  in $G$ with leaf vector $(x_1,\ldots,x_s)$ and length vector $(j_1,\ldots,j_s)$ is less than $f(j,L)$ where $j=j_1+\cdots+j_s$. If $S$ is admissible but not light, then we say that it is heavy.
\end{Definition}

At this point, let us say a few words about the function $f(j,L)$ given in Definition \ref{def:f-definition},
as this function plays an important role in our arguments.
In application we always assume that the parameter $L$ is sufficiently larger than $s,t,k$ and $K$
and roughly speaking $f(j,L)$ is chosen so that $f(j,L)\gg f(j-1,L)$, i.e., 
$\frac{f(j,L)}{f(j-1,L)}\rightarrow\infty$ as $L\rightarrow\infty$.

Next, we give several lemmas. Lemma \ref{lem:independent-paths}
is similar to one used in \cite{CJL}. Lemma \ref{lem:independent-spiders} has
its analogous counterparts in \cite{JQ} and \cite{Janzer2}. However, since our terminologies and choices
of constants are slightly different, we include full proofs for completeness.

\begin{lem} \label{lem:admissible}
Let $G$ be a $K$-almost-regular graph. Let $1\leq i\leq j$ be integers. Let $x,w,y$ be vertices in $G$.
Then the number of $j$-admissible paths in $G$ that have $x,w,y$ as the first, $(i+1)$-th and last vertices, respectively
is at most $f(i,L)\cdot f(j-i,L)$. Furthermore, if $i=1$ or $j$, then there are most $f(j-1,L)$ such paths.
\end{lem}
\begin{proof}
Let $\cP$ be  the family of $j$-admissible paths in $G$ that have $x,w,y$ as the first, $(i+1)$-th, and last vertices, respectively.
Let $P\in \cP$, by definition, each proper subpath of $P$ is light. So $P$ is the union of $i$-light path from $x$ to $w$ and
a $(j-i)$-light path from $w$ to $y$. By definition of light paths there are at most $f(i,L)$ $i$-light paths in $G$ with
ends $x,w$ and at most $f(j-i,L)\cdot (j-i)$-light paths with ends $w$ and $y$. So $|\cP|\leq f(i,L)\cdot f(j-i,L)$.

If $i=1$ then every $P\in \cP$ is the union of the edge $xw$ and a $(j-1)$-light path with ends $w$ and $y$. So $|\cP|\leq f(j-1,L)$.
The case $i=j$ is similar.
\end{proof}

\begin{lem}\label{lem:independent-paths}
	Let $x,y$ be two vertices and $\cC$ be family of $j$-admissible paths between $x$ and $y$. Then  there are 
	$|\cC|/[{j^2 \cdot f(j-1,L)^2}]$ members of $\cC$ that are pairwise
	vertex disjoint outside $\{x,y\}$.
\end{lem}
\begin{proof}
	Let $\cC'=\{Q_1,\ldots,Q_r\}\subseteq \cC$ be a maximal subfamily of $\cC$ that are pairwise vertex disjoint outside $\{x,y\}$. Let $W=\bigcup_{j=1}^r V(Q_i)\setminus\{x,y\}$. Then $|W|=(j-1)r$. By maximality, every member of $\cC$ must contain a vertex $v\in W$ as an internal vertex. For each $v\in W$ and each $1\leq i\leq j-1$ let $\cC_{v,i}$ be the subfamily of members of
	$\cC$ that contains $v$ as its $(i+1)$-th vertex (when the member is viewed from $x$ to $y$). Then $\cC=\bigcup_{v,i} \cC_{v,i}$. By Lemma \ref{lem:admissible},  for any fix $v,i$, we have $|\cC_{v,i}|\leq f(i,L)\cdot f(j-i,L)\le f(j-1,L)^2$. Hence
	\begin{equation*}
|\cC|=|\bigcup_{v,i}\cC_{v,i}|\leq \sum_{v\in W}\sum_{i=1}^{j-1} f(j-1,L)^2 <rj^2 f(j-1,L)^2.
	\end{equation*}
Solving the inequality for $r$, we get the desired claim.
\end{proof}

For two spiders with the same leaf vector and length vector, we say they are \emph{internally disjoint} if they are vertex disjoint outside their leaves.
\begin{lem}\label{lem:independent-spiders}
	Let $\cS$ be a family of admissible spiders with leaf vector $(x_1,\ldots,x_s)$ and length vector $(j_1,\ldots,j_s)$. Then among  there are 
	$|\cS|/[j^2 \cdot f(j-1,L)^2]$
members of $\cS$ that are pairwise vertex disjoint outside $\{x_1,\ldots, x_s\}$, where $j=j_1+\cdots+j_s$.
\end{lem}
\begin{proof}
	Let $\cS'=\{S_1,\ldots,S_r\}\subseteq \cC$ be a maximal subfamily of members of $\cS$ that are pairwise vertex disjoint outside $\{x_1,\ldots, x_s\}$. Let $W=\bigcup_{i=1}^r V(S_i)\setminus \{x_1,\ldots, x_s\}$. Then $|W|=(j-s+1)r$. 
By maximality of $\cS'$, every member of $\cS$ must contain some $v\in W$ as a non-leaf vertex.  
For each $v\in W$ let $\cD_v$ denote the subfamily of members of $\cS$ that contain $v$ as the center.
For each $v\in W$, $i\in [s]$, and $1\leq \ell <j_i$, let $\cS_{v,i,\ell}$ denote the subfamily of members of $\cS$ in which
$v$ is on the $i$-th leg  and the distance from $v$ to $x_i$ is $\ell$.
Then $\cS=(\bigcup_{v\in W} \cD_v) \cup (\bigcup_{v\in W, i\in [s], 1\leq \ell<j_i} \cS_{v,i,\ell})$.

Let $S\in \cD_v$. Then by definition, for each $i\in[s]$, the $i$-th leg of $S$ is a $j_i$-light path between $v$ and $x_i$.
Hence, by the definition of light paths.
\[|\cD_v|\leq \prod_{i=1}^s f(j_i, L)\le f(j-1,L)^2,\]
where the last inequality holds because by Definition \ref{def:f-definition} we have that $\prod_{i=1}^s f(j_i, L)\le f(j_1+j_2-1, L)^2\prod_{i=3}^s f(j_i, L)\le f(j_1+j_2, L)\prod_{i=3}^s f(j_i, L)\le\cdots\le f(j_1+\cdots+j_{s-1},L)f(j_s,L)\le f(j-1,L)^2$.

Next, fix $v\in W$, $i\in [s]$, and $1\leq \ell <j_i$. Let $S\in \cS_{v,i,\ell}$. Since $S$ is admissible,  the $v,x_i$-path in $S$ is $\ell$-light while the rest of $S$ is an $s$-legged proper sub-spider,
which by definition, is light. This implies that
\[|\cS_{v,i,\ell}|\leq f(\ell,L)f(j-\ell, L)\le f(j-1,L)^2.\]

	Putting everything together, we obtain
	\begin{equation*}
	|\cS|\leq |W| f(j-1,L)^2+\sum_{v\in W}\sum_{i=1}^{s}\sum_{\ell=1}^{j_i-1} f(j-1,L)^2,
	\end{equation*}
	which implies that $|\cS|\leq rj^2 f(j-1,L)^2$, from which the claim follows.
\end{proof}

The following lemma is proved in \cite{JQ}. A spider has {\it height} $\ell$ if all of its legs have length $\ell$.

\begin{lem} \label{lem:spider-shrinking2} {\rm (\cite{JQ} Lemma 3.6)}
	Let $G$ be a $K$-almost-egular graph with minimum degree $\delta$.
	Let $x$ be a vertex. Let $\mathcal{C}$ be a family
	of paths of length $h$ with one end $x$  and
	another end in a set $S$. For each $i\in [h]$ there exists a vertex $x_i$ and a  spider of height $i$ with center $x_i$
	and leaves in $S$ which has at least  $|\cC|/[h (K\delta)^{h-1}]$ legs. Furthermore, $x_i=x$ if and only if $i=h$.
\end{lem}

We also need a standard averaging lemma as below.

\begin{lem} \label{lem:common-degree}
Let $0<c<1$ be a real and  $m$ be a positive integer.
Let $G$ be a bipartite graph with a bipartition $(X,Y)$.
Suppose that $e(G)\geq c|X||Y|$  and that $c|Y|\geq 2m$.
Then there exists an $m$-set $S$ in $Y$ such that
$|N^*_G(S)|\geq (c/2)^m |X|$.
\end{lem}
\begin{proof}
By our assumption, the average degree of vertices in $X$ is at least $c|Y|$. Let $\cF$ be the family of 
$K_{1,m}$'s with center in $X$. Then $|\cF|=\sum_{x\in X} \binom{d_G(x)}{m}\geq |X|\binom{c|Y|}{m}$,
where the last inequality uses the convexity of the function $\binom{x}{m}$.
Hence, by averaging there exists an $m$-set $S$ in $Y$ such that the number of members of $\cF$
that have $S$ as the leaf set is at least 
\[|X|\frac{\binom{c|Y|}{m}}{\binom{|Y|}{m}} \ge |X| \left(\frac{c|Y|-m}{|Y|-m}\right )^m>(c/2)^m|X|, \]
where the last inequality uses the condition $c|Y|\geq 2m$.
\end{proof}

Finally, we need a standard cleaning lemma.
\begin{lem} \label{lem:min-degree}
If $B$ is a bipartite graph with parts $X$ and $Y$, then it has subgraph $B'$ such that $e(B')\geq \frac{e(B)}{2}$ and
$\forall x\in X\cap V(B'), d_{B'}(x)\geq \frac{e(B)}{4|X|}$ and $\forall y\in Y\cap V(B'), d_{B'}(y)\geq \frac{e(B)}{4|Y|}$.
\end{lem}
\begin{proof} 
Whenever there is a vertex in $X$ whose degree becomes less than $\frac{e(B)}{4|X|}$ or a vertex in $Y$ whose degree becomes less than $\frac{e(B)}{4|Y|}$, we delete it. Let $B'$ denote the final subgraph of $B$. As the number of edges deleted is at most $|X|\cdot \frac{e(B)}{4|X|}+|Y|\cdot \frac{e(B)}{4|Y|}=\frac{e(B)}{2}$, $e(B')\geq \frac{e(B)}{2}$. By definition, $B'$ satisfies our requirements.
\end{proof}

\section{Proof of Theorem \ref{thm:regular}}

\subsection{Overall structure of the proof}

Our overall strategy has roots in  the work of Conlon and Lee \cite{CL} and the work of Conlon, Janzer, and Lee \cite{CJL},
particularly \cite{CJL}. Some of the strategies used there were later augmented (through the concepts of admissible, light, and heavy spiders) in the work of Jiang and Qiu \cite{JQ} and the work of  Janzer \cite{Janzer2}.  In particular, Janzer \cite{Janzer2} introduced a creative way to extending spiders, an idea that we will develop further. Overall, our proof combines ideas from
\cite{CJL}, \cite{JQ},\cite{Janzer2} and some new ideas.

Let $G$ be a $K$-almost-regular $t*S^s_{b,k}$-free graph on $n$ vertices, where $n$ is sufficiently large.
To the prove the theorem, it suffices to show that 
there exists a constant $C$ depending on $s,b,k$ such that if $\delta(G)\geq Cn^{\frac{s-1}{(s-1)k+b}}$, then $G$ must contain a copy of $t*S^s_{b,k}$, which would contradict $G$ being $t*S^s_{b,k}$-free and complete the proof. The general strategy is to show that (1) $G$ contains many copies of $S^s_{b,k}$ and (2) most of these copies of $S^s_{b,k}$ are light.
Then by averaging, there exist some vector $(x_1,\ldots, x_s)$ of $s$ vertices which is the leaf vector of a large number of light copies of $S^s_{b,k}$. This will imply that all these spiders are heavy, giving us contradiction.
More specifically, the proof of Theorem \ref{thm:regular} follows readily after we establish the following two crucial lemmas.

\begin{lem}\label{lem:heavy-paths}
	Let $G$ be a $t*S^s_{b,k}$-free $K$-almost-regular graph on $n$ vertices with minimum degree $\delta=\omega(1)$. Then provided that $L$ is sufficiently large compared to $s,t,k,K$, for any $2\le j\le k$, the number of $j$-heavy paths in $G$ is at most $\frac{(j+1)^{j+1}}{L}n\delta^{j}$.
\end{lem}

\begin{lem}\label{lem:heavy-spiders}
	Let $G$ be a $t*S^s_{b,k}$-free $K$-almost-regular graph on $n$ vertices with minimum degree $\delta=\omega(1)$.
	Let $1\le j_1\le b$ and $1\le j_2,\dots, j_s\le k$ be integers.
	Then provided that $L$ is sufficiently large compared to $s,t,k,K$, the number of heavy spiders with length vector $(j_1,\ldots, j_s)$ 
	is at most $\frac{27K^{j-2}}{L}n\delta^j$ where $j=j_1+\cdots+j_s$.
\end{lem}

We now show how Theorem \ref{thm:regular} follows from Lemma \ref{lem:heavy-paths}
and Lemma \ref{lem:heavy-spiders}.

\medskip

\noindent{\bf Proof of theorem \ref{thm:regular}:}
Let $L$ be a sufficiently large constant compared to $s,t,k,K$. Let $G$ be a $K$-almost-regular $t*S^s_{b,k}$-free graph on $n$ vertices with minimum degree $\delta$. Let $h=e(S^s_{b,k})=(s-1)k+b$. Suppose to the contrary that $\delta \geq C n^{\frac{s-1}{(s-1)k+b}}$,
where $C:=2f(h,L) (h+1)!$.
Let $\cS$ be the family of spiders in $G$ with length vector $(b,k,\ldots,k)$. By a greedy process, it is easy to see that
\begin{equation*}
|\cS|\ge\frac{1-o(1)}{(h+1)!}\cdot n\delta^{h}
\end{equation*}
Let $\cS_1$ be the family of spiders in $\cS$ that contain some heavy path of length $2\le j\le k$. As the maximum degree of $G$ is at most $K\delta$, by Lemma \ref{lem:heavy-paths}, we have
\begin{equation*}
|\cS_1|\le\sum_{j=2}^k \binom{h}{j}\frac{(j+1)^{j+1}}{L}n\delta^j (K\delta)^{h-j}\le\frac{(k+1)^{k+2}K^{sk}h!}{L}n\delta^{h},
\end{equation*}
where the factor $\binom{h}{j}$ upper bounds the number of positions of a $j$-heavy paths in $S_{b,k}^s$. Let $\cS_2$ be the family of spiders in $\cS$ that contain some $s$-legged heavy sub-spider. As the maximum degree of $G$ is at most $K\delta$, by Lemma \ref{lem:heavy-spiders}, we have 
\begin{equation*}
|\cS_2|\le \sum_{\substack{1\le j_1\le b\\ 1\le j_2,\ldots,j_s\le k}}\frac{27K^{j_1+\cdots+j_s-2}}{L}n\delta^{j_1+\cdots +j_s}\cdot(K\delta)^{h-(j_1+\cdots +j_s)}\le\frac{27K^{sk}k^s}{L}n\delta^{h}.
\end{equation*}
Let $\cS'=\cS-(\cS_1\cup\cS_2)$. Then it follows that
\begin{equation*}
|\cS'|\ge|\cS|-(|\cS_1|+|\cS_2|)\ge \frac{1-o(1)}{(h+1)!}\cdot n\delta^{h}-\frac{K^{sk}(27k^s+(k+1)^{k+2}h!)}{L}n\delta^{h}\ge \frac{n\delta^{h}}{2(h+1)!},
\end{equation*}
where the last inequality holds since $L$ is sufficiently large. As $\delta\ge Cn^{\frac{s-1}{(s-1)k+b}}$, and $C=2(h+1)!f(h,L)$, it follows that $|\cS'|\ge f(h,L)n^s$. By averaging, there exists an $s$-tuple $(x_1,\ldots,x_s)$ of distinct vertices, such that the sub-family $\cS''$ which consists of all spiders in $\cS'$ with leaf vector $(x_1,\ldots,x_s)$ has size $|\cS''|\ge f(h,L)$. For any $S\in\cS''$, since $S$ contains no heavy path of length at most $k$, every leg of $S$ is light. Since $S$ does not contain any $s$-legged heavy sub-spider, $S$ is light. So $\cS''$ is a family of at least $f(h,L)$ light spiders with leaf vector $(x_1,\ldots,x_s)$ and length vector $(b,k,\ldots,k)$. This contradicts the definition of the light spider with length vector $(b,k,\ldots,k)$.
\qed

\medskip

Thus, to complete our proof of Theorem \ref{thm:regular}, it remains to prove Lemma \ref{lem:heavy-paths}
and Lemma \ref{lem:heavy-spiders}.
Lemma \ref{lem:heavy-spiders} was proved by Janzer for the case $b=k$ in details in \cite{Janzer2}. It was pointed out in 
the concluding remarks of \cite{Janzer2} (Lemma 4.3) that the same proof works in more general settings (including
the one for our Lemma \ref{lem:heavy-spiders}). To make our paper self-contained, we include a sketch of a proof of Lemma \ref{lem:heavy-spiders} in the appendix, following Janzer's arguments. As the author of \cite{Janzer2} pointed out the main obstacle to proving Conjecture \ref{conj:Janzer} is to establish analogous statements for heavy paths. Indeed, the method developed in \cite{CJL} (and later used in \cite{JQ} and \cite{Janzer2}) for heavy paths is not applicable in the new setting.  

Our main contribution in this paper is to develop a method to handle heavy paths for $t*S^s_{b,k}$-free graphs, resulting in Lemma \ref{lem:heavy-paths}. 
We believe that some of the ideas we developed here can be further expanded to potentially yield further progress
on Conjecture \ref{conj:exponent} and Conjecture \ref{conj:Janzer}.

\subsection{Building $t*S^s_{b,k}$ using heavy paths}
The rest of the section is devoted to proving Lemma \ref{lem:heavy-paths}.
The proof consists of two parts: the case of $j>\frac{k+b}{2}$ (Lemma \ref{lem:long}) and the case of $2\le j\le\frac{k+b}{2}$
(Lemma \ref{lem:short}).

\subsubsection{Long heavy paths: the $j>\frac{k+b}{2}$ case}

\begin{lem}\label{lem:long}
	Let $G$ be a $t*S^s_{b,k}$-free $K$-almost-regular graph on $n$ vertices with minimum degree $\delta=\omega(1)$. Then provided that $L$ is sufficiently large compared to $s,t,k,K$, for any $\frac{k+b}{2}< j\le k$, the number of $j$-heavy paths in $G$ is at most $\frac{n\delta^{j}}L$.
\end{lem}
\begin{proof} We define some constants as follows. Let 
\[c_1=\frac{1}{4Lf(j-1,L)^2 }, \, c_2=\frac{c_1}{4K^b},  \, c_3=\frac{c_1}{4K^{j-b}},  \, c_4=\frac{c_3}{bK^{b-1}}, \,
c_5=\frac{c_2}{K^{j-b}},\, c_6=(\frac{c_5}{2})^t\cdot c_2.\]

Suppose to the contrary that the number of $j$-heavy paths is at least $\frac{n\delta^{j}}L$. By averaging, there exists a vertex $w$ such that the family $\cP_w$ consisting of all the $j$-heavy paths of the form $xx_1\cdots x_{b-1} w x_{b+1}\cdots x_{j-1}y$ has size at least $\frac{\delta^{j}}{L}$. Let $X$ be the set of vertices in $G$  that play the role of $x$ in some member of $\cP_w$ and
$Y$ the set of vertices in $G$ that play the role of $y$ in some member of $\cP_w$.  Then  $X\subseteq \Gamma_b(w)$ and $Y\subseteq \Gamma_{j-b}(w)$.
Since $G$ is $K$-almost-regular and thus has maximum degree  at most $K\delta$, we have

\begin{equation}\label{eq:XY}
|X|\le (K\delta)^b \mathrm{~~and~~} |Y|\le (K\delta)^{j-b}.
\end{equation}

Note that $X,Y$ may not be disjoint.
We define an auxiliary graph $B$ on $X\cup Y$, such that $\forall x\in X, y\in Y$,  $xy\in E(B)$ if and only if
some member $P$ of $\cP_w$ have ends $x$ and $y$.

\medskip

{\bf Claim 1.} For every $x\in X$ there is a $(x,w)$-path of length $b$ in $G$. For every $y\in Y$ there is an $(w,y$)-path
of length $j-b$ in $G$.
For all $x\in X, y\in Y$ such that $xy\in E(B)$ there exist at least $L$ internally disjoint $x,y$-paths of length $j$ in $G$. 

\medskip

{\it Proof of Claim 1.}  The first two statements follow from the definitions of $X$ and $Y$.
Suppose $x\in X, y\in Y$ and $xy\in E(B)$. By definition, some member $P\in \cP$ has $x,y$ as ends.
By the definition of $\cP$, $P$ is $j$-heavy and thus there exist at least $f(j,L)$ many $j$-admissible paths with ends $x$ and $y$ in $G$. By Lemma \ref{lem:independent-paths} among them we can find at least
\[f(j,L)/[{j^2 f(j-1,L)^2}]\geq L \] 
that are pairwise vertex disjoint outside $\{x,y\}$, where the inequality holds by Proposition \ref{prop}.
\qed

\medskip

For any fixed $x\in X$ and $y\in Y$, by Lemma \ref{lem:admissible} there are at most $f(b,L)\cdot f(j-b,L)$ members of $\cP_w$ that have
ends $x$ and $y$. Hence

\[e(B)\ge\frac{|\cP_w|}{f(b,L)f(j-b,L)}\ge\frac{ \delta^{j}}{Lf(j-1,L)^2}.\]

Now, let us color each vertex in $X\cup Y$ with color $1$ or $2$ independently at random with probability $\frac12$ each.
Let $X_1$ denote the set of vertices in $X$ that receive color $1$ and $Y_2$ the set of vertices in $Y$ that receive color $2$.
Let $\widetilde{B}$ denote the subgraph of $B$ consisting of edges that join a vertex in $X_1$ to a vertex in $Y_2$.
Each edge of $B$ has probability at least $1/4$ of being in $\widetilde{B}$. Hence there exists a coloring such that
the resulting $\widetilde{B}$ has at least $(1/4)|B|$ edges. 
Then $\widetilde{B}$ is bipartite with parts $X_1$ and $Y_2$ and by our discussion
\begin{equation} \label{eq:Btilde-lower}
e(\widetilde{B})\ge  \frac{\delta^{j}}{4Lf(j-1,L)^2}=c_1 \delta^j.
\end{equation}

By Lemma \ref{lem:min-degree}, $\widetilde{B}$ contains a subgraph $B'$
with parts $X'\subseteq X_1$ and $Y'\subseteq Y_2$ such that
\begin{equation} \label{eq:min-degree1}
  \forall x\in X', d_{B'}(x)\ge \frac{e(\widetilde{B})}{4|X_1|}\geq \frac{c_1}{4K^b} \delta^{j-b} =c_2\delta^{j-b},  
 \end{equation}
 and
 \begin{equation}\label{eq:min-degree2}
 \forall y\in Y', d_{B'}(y)\geq \frac{e(\widetilde{B})}{4|Y_2|}\geq \frac{c_1}{4K^{j-b}} \delta^b =c_3\delta^b.
 \end{equation}
 By \eqref{eq:min-degree1} and \eqref{eq:min-degree2},
 \[|X'|\geq c_3 \delta^b, \mbox{ and }|Y'|\geq c_2\delta^{j-b}.\]
Since $X'\subseteq X$, by Claim 1,  there are at least $|X'|$ paths of length $b$ with one end $w$ and another end in $X'$. By Lemma \ref{lem:spider-shrinking2}, there exists a spider $T$ of height $b$ with center $w$ and leaves in $X'$ whose number of legs is at least $\frac{|X'|}{b(K\delta)^{b-1}}\geq \frac{c_3}{bK^{b-1}} \delta =c_4\delta$.

Let $X''\subseteq X'$ be the leaf set of this spider. Then $|X''|\ge c_4\delta$.
Let $B''$ be the subgraph of $B'$ induced by $X''\cup Y'$. By \eqref{eq:XY} and \eqref{eq:min-degree1} 

\[e(B'')\ge c_2\delta^{j-b}|X''|\geq \frac{c_2}{K^{j-b}}|X''||Y'|=c_5|X''||Y'|.\]
Since $|X''|\geq c_4\delta$ and $\delta=\omega(1)$, for sufficiently large $n$ we may assume that 
$c_5 |X''|\geq 2t$. By Lemma \ref{lem:common-degree}
 \begin{equation} \label{eq:large-common}
 \exists X_0\subseteq X'' \mbox{ such that } |X_0|=t \mbox{ and }   |N^*_{B''}(X_0)|\geq (c_5/2)^t |Y'|.
\end{equation}
Now, let us fix a $t$-set $X_0\subseteq X''$ guaranteed in \eqref{eq:large-common}.
 Let $T_0$ be the sub-spider of $T$ with leaf set $X_0$. Let $Y''=N^*_{B''}(X_0)$. 
  Then \[|Y''|\ge(c_5/2)^t |Y'|\geq (c_5/2)^t c_2\delta^{j-b} =c_6 \delta^{j-b}.\]  
  
  Let $\cC$ be the family of paths of length $j-b$ with one end $w$ and another end in $Y''$. By Claim 1, 
  $|\cC|\geq |Y''|\geq c_6 \delta^{j-b}$.
  Since $G$ has maximum degree at most $K\delta$, for any vertex $u\neq w$, the number of paths in $\cC$ that contain $u$ is at most $(K\delta)^{j-b-1}$. Let $\cC_1$ be the family of paths in $\cC$ that is vertex-disjoint from $V(T_0)-\{w\}$. Then $$|\cC_1|\ge|\cC|-(|V(T_0)|-1)(K\delta)^{j-b-1}\geq c_6\delta^{j-b}-bt(K\delta)^{j-b-1}\geq (c_6/2) \delta^{j-b},$$ 
where the last inequality holds for sufficiently large $n$ because $\delta=\omega(1)$. 
As $j>\frac{k+b}{2}$, we have $k-j<j-b$. Applying Lemma \ref{lem:spider-shrinking2} to $\cC_1$, as $\delta=\omega(1)$, there exists a $t$-legged spider $T_1$ of height $k-j$ with center $v_1\neq w$ and leaf set $Y_1\subseteq Y''$. Note that $V(T_0)\cap V(T_1)=\emptyset$. 
Using the same strategy, we can find $s-1$ vertex-disjoint $t$-legged spiders $T_1,\ldots,T_{s-1}$ of height $k-j$ one by one, with $T_i$'s center $v_i\neq w$ and leaf set $Y_i\subseteq Y''$, such that $V(T_i)\cap V(T_0)=\emptyset$.

Suppose $X_0=\{x_1,\dots, x_t\}$. For each $i\in [s-1]$, suppose $Y_i=\{y_i^1,\ldots, y_i^t\}$.
Let $Y_0=\bigcup_{i=1}^{s-1}Y_i$. Then
$Y_0\subseteq Y''=N^*_{B''}(X_0)$. Hence, $\forall x\in X_0, y\in Y_0$, $xy\in E(B'')\subseteq E(B)$
and  by Claim 1 there exist at least $L$ internally disjoint paths of length $j$ joining $x$ and $y$. 
As $L$ is a sufficiently large constant, we can greedily find $t(s-1)$ paths $P_{i,\ell}$ of length $j$, such that
for any $i\in [t]$ and $\ell\in [s-1]$, $P_{i,\ell}$ has ends $x_i$ and $y_\ell^i$ and contains no vertex of 
$\bigcup_{i=0}^{s-1} V(T_i)\cup (X_0\setminus\{x_i\})$ and that the $P_{i,\ell}$'s are pairwise vertex disjoint outside $X_0$.
Now, $(\bigcup_{i=0}^{s-1} T_i) \cup (\bigcup_{i\in [t], \ell\in [s-1]} P_{i,\ell})$ forms a copy of $t*S^s_{b,k}$ in $G$, a contradiction. This completes our proof.
\end{proof}


\subsubsection{Short heavy paths: the $2\le j\le\frac{k+b}{2}$ case}

This subsection handles the most difficult part of our main proof and is where most of the new ideas are used.

\begin{lem}\label{lem:short}
	Let $G$ be a $t*S^s_{b,k}$-free $K$-almost-regular graph on $n$ vertices with minimum degree $\delta=\omega(1)$. Then provided that $L$ is sufficiently large compared to $s,t,k,K$, for any $2\le j\le\frac{k+b}{2}$, the number of $j$-heavy paths is at most $\frac{(j+1)^{j+1}}{L} n\delta^{j}$.
\end{lem}

We break the proof of Lemma \ref{lem:short} into several steps. The general strategy is to show that if the family $\cF$ of $j$-heavy paths is too large then we find a copy of $t*S^k_{b,k}$ in $G$, which is a contradiction. We start by doing some cleaning to $\cF$ in order to set up further arguments. Before that, let us set some constants to be used throughout the subsection.

\begin{Definition} \label{def:constants}
	Let $D=2K^jL (f(j-1,L)^2$ and  $M=D^{s+1}$.
\end{Definition}
Comparing Definition \ref{def:f-definition} and Definition \ref{def:constants}, we see that
\begin{equation} \label{eq:f-facts}
f(j,L)=10j^4 D^{s+3} = 10j^4 D^2 \cdot M. 
\end{equation}

Now we introduce our cleaning lemma. Given a path $P=v_0v_1\cdots v_j$ and $0\leq i<j$, we define the {\it initial $i$-segment} of $P$ to be the subpath $v_0v_1\cdots v_i$. 

\begin{lem}\label{lem:cleaning1}
	Let $G$ be a $K$-almost-regular graph on $n$ vertices with minimum degree $\delta=\omega(1)$. Suppose that the number of $j$-heavy paths is at least $\frac{(j+1)^{j+1}}{L} n\delta^{j}$. Then there exist a vertex $w$, vertex disjoint sets $A_0,\dots, A_j$ and a family $\cF$ of $j$-heavy paths with $\bigcup_{i=0}^j A_i=\bigcup_{P\in\cF}V(P)$ satisfying
\begin{enumerate}
\item $A_0\subseteq \Gamma_1(w) \mbox{ and } A_j\subseteq \Gamma_{j-1} (w)$. 
\item Each member of $\cF$ has the form $v_0v_1\cdots v_j$ where $\forall i\in \{0,1,\ldots, j\}, v_i\in A_i$.
\item There exists a set $V_0$ with $A_0\subseteq V_0\subseteq \Gamma_1(w)\setminus A_j$ such that  for every $y\in A_j$, there are at least $\frac{|V_0|}{D}$  many $x\in V_0$ such that $x,y$ are ends of a heavy $j$-path in $G$. Furthermore,
$|V_0|\geq (2K/D)\delta$.
\item For each $x\in A_0$, there are at least $M$ vertices $y\in A_j$ such that $x,y$ are ends of at least $DM$ members of $\cF$. 
For each $y\in A_j$, there are at least $M$ vertices $x\in A_0$ such that $x,y$ are ends of at least $DM$ members of $\cF$. 
\item For each $P\in \cF$ and $0\leq i<j$, the initial $i$-segment of $P$ is contained in at least $jM(K\delta)^{j-i-1}$ members of $\cF$.
\end{enumerate}
\end{lem}	
\begin{proof}
 Let $\cC$ be the collection of all $j$-heavy paths in $G$. By our assumption, $|\cC|\geq \frac{(j+1)^{j+1}}{L}n\delta^j$.
 Let us independently color each vertex of $G$ with a color in $\{0,1,\ldots, j\}$, with each color chosen uniformly at random. For each $0\leq i\leq j$, let $V_i'$ denote the set of vertices in $G$ receiving color $i$.
For any $j$-heavy path $P=v_0v_1\cdots v_j$, call $P$ {\it good} if $\forall 0\leq i\leq j, v_i\in V_i'$. Let $\cC'$ denote the family of
all good heavy $j$-paths. Clearly each $j$-heavy path in $G$ is good with probability $(\frac{1}{j+1})^{j+1}$. So there exists a vertex coloring for which
\[|\cC'|\geq \frac{|\cC|}{(j+1)^{j+1}}\geq \frac{n\delta^j}{L}.\]
Let us fix such a coloring and the corresponding $\cC'$.
 
	By averaging, there exists a vertex $w$ such that subfamily $\cP_w$ of members of $\cC'$ of the form $v_0wv_2\cdots v_j$ has size at least $|\cP_w|\ge\frac{\delta^{j}}{L}$. For each $i\in\{0,1,\ldots,j\}$, let $V_i$ be the set of vertices in $V_i'$ that are contained in members of $\cP_w$.
	By our definitions, $V_0\subseteq \Gamma_1(w)$ and $V_j\subseteq \Gamma_{j-1}(w)$.
	Since $G$ has maximum degree at most $K\delta$, we have
	\begin{equation} \label{eq:XY-bounds}
	|V_0|\le K\delta \mbox{ and } |V_j|\le {(K\delta)^{j-1}}.
	\end{equation}
	
	Let $B$ denote the auxiliary bipartite graph with a bipartition $(V_0,V_j)$ such that $\forall x\in V_0, y\in V_j, xy\in E(B)$ if and only if $x,y$ are ends of some member of $\cP_w$.  For each $xy\in E(B)$ with $x\in V_0, y\in V_j$, let 
	$\cP_{xy}$ be the subfamily of members of $\cP_w$ that cover $x,y$ and let $\cJ_{xy}$
	be the family of $j$-heavy paths in $G$ that have $x,y$ as ends.

\medskip

{\bf Claim 1.} For each $xy\in E(B)$, we have $1\leq |\cP_{xy}|\leq f(j-1,L)$ and $|\cJ_{xy}|\geq f(j,L)$.

\medskip

{\it Proof of Claim 1.} Let $xy\in E(B)$, with $x\in V_0, y\in V_j$. That $|\cP_{xy}|\geq 1$ is clear.
Let $P\in \cP_{xy}$. By definition $P$ is $j$-admissible and $P=xw\cup Q$, where $Q$ is a $(w,y)$-path of length $j-1$.
Since $P$ is admissible, $Q$ is $(j-1)$-light. So the number of possible $Q$ in $G$ is at most $f(j-1,L)$. So, $|\cP_{xy}|\leq f(j-1,L)$.
Next, since $P$ is a $j$-heavy path in $G$ with ends $x,y$, by definition, $G$ contains
at least $f(j,L)$ $j$-heavy paths with ends $x,y$. So $|\cJ_{xy}|\geq f(j,L)$.
\qed

\medskip

By Claim 1 and Definition \ref{def:constants}

	\begin{equation}\label{eq:B0-lower}
	e(B)\ge\frac{|\cP_w|}{f(j-1,L)}\ge\frac{\delta^j}{Lf(j-1,L)}\ge (2K^j/D)\delta^j.
	\end{equation}
Since $|V_j|\leq (K\delta)^{j-1}$, \eqref{eq:B0-lower} implies 
\begin{equation} \label{eq:V0-lower}
|V_0|\geq e(B)/|V_j|\geq (2K/D)\delta.
\end{equation}
	Let $V^*_j$ be the set of $y\in V_j$ for which $d_B(y)\ge \frac{|V_0|}{D}$. 
	Let $B^*$ denote the subgraph of $B$ induced by $V_0\cup V^*_j$. Then 
	\begin{equation} \label{eq:B*lower}
	e(B^*)\ge{e(B)-|V_0||V_j|/D}\ge (2K^j/D)\delta^j-(K\delta)(K\delta)^{j-1}/D= K^j\delta^j/D.
	\end{equation}

For each $xy\in E(B^*)$, we have
\[ |\cJ_{xy}|\geq f(j,L)\ge10j^2 D^2M,\]
where the last inequality holds by \eqref{eq:f-facts}. Let $\cJ'_{xy}$ be a subfamily of $\cJ_{xy}$ of size exactly $10j^2D^2M$. Let
\begin{equation} \label{eq:F0-definition}
\cF_0=\bigcup_{xy\in E(B^*)} \cJ'_{xy}.
\end{equation}
Then by \eqref{eq:B*lower}
\begin{equation} \label{eq:F0-lower}
|\cF_0|= e(B^*)\cdot 10j^2D^2M\geq10j^2 D M K^j\delta^j.
\end{equation}
We next obtain $\cF$ from $\cF_0$ through some further cleaning. Initially let $\cF=\cF_0$.
Throughout the process, for each $x\in V_0, y\in V^*_j$
let $\lambda(x,y)$ denote the number of remaining members of $\cF$ that have ends $x,y$. We update
the function $\lambda(x,y)$ automatically after each removal.
Whenever is a vertex $x\in V_0$ such that the number of $y\in V^*_j$ with $\lambda(x,y)\geq DM$ is less than $M$
(which we refer to as {\it $x$ becomes small}), remove all the members of $\cF$ that contain $x$. 
 Similarly, whenever there is a vertex $y\in V^*_j$ such
that the number of $x\in X$ with $\lambda(x,y)\geq DM$ is less than $M$
(which we refer to as {\it $y$ becomes small}), remove all the members of $\cF$ that contains $y$.
Whenever there is a member $P\in \cF$ (viewed as a path from $V_0$ to $V^*_j$)
contains an initial $i$-segment $I$, for some $0\leq i<j$, that is contained is less than $jM(K\delta)^{j-i-1}$ members of $\cF$
we remove all the members of $\cF$ containing $I$. We continue the process until no further removal can be performed.

The number of members of $\cF$ we removed for each $x\in V_0$ that becomes small is at most
\[(|V^*_j|-M) M + M (10j^2D^2 M) \leq 2M(K\delta)^{j-1},\]
for sufficiently large $n$, since $\delta=\omega(1)$. Similarly, 
the number of members of $\cF$ that we removed for each vertex $y\in V^*_j$ that becomes small is at most
\[(|V_0|-M)M +M(10j^2D^2M)\leq 2M(K\delta).\]
So, the total number of members of $\cF$ we removed due to either a vertex in $X$ becoming small or a vertex
in $Y^*$ becoming small is at most 
\[|V_0|\cdot 2M(K\delta)^{j-1} + |V^*_j| \cdot2M (K\delta) \leq 4M(K\delta)^j.\]
The number of members of $\cF$ that we removed due to some initial segment is contained in too few members
is at most 
	\begin{equation*}
	\sum_{i=0}^{j-1}|V_0|(K\delta)^i\cdot jM(K\delta)^{j-i-1}\le j^2M(K\delta)^j.
	\end{equation*}
Combining the above two inequalities, the total number of members of $\cF$ that
we removed is at most
	\begin{equation*}
	(j^2+4)M(K\delta)^j\le 5j^2MK^j\delta^j\le \frac{|\cF_0|}2.
	\end{equation*} 
So in particular, the final $\cF$ is nonempty.

Now, for each $0\leq i\leq j$,  let $A_i$ be the set of vertices in $V_i$ that are contained in members of the final $\cF$.
In particular, note that $A_j\subseteq V^*_j$. 
Let us check that $w,A_0,\ldots, A_j$ and $\cF$ satisfy the five conditions of the lemma.
Condition 1 and condition 2 clearly hold by our discussion so far. 
Condition 3 holds since $A_j\subseteq V^*_j$
and each vertex $y\in V^*_j$ satisfies $d_B(y)\geq \frac{|V_0|}{D}$ and $|V_0|\geq (2K/D)\delta$ by \eqref{eq:V0-lower}.
Conditions 4 and 5 hold due to our cleaning rules. This completes the proof of the lemma.
\end{proof}

\begin{lem}\label{lem:cleaning2}
	Let $G,A_0,\ldots, A_j$ and $\cF$ be as stated in Lemma \ref{lem:cleaning1}. Then
\begin{enumerate}
\item For any $0\le i<j$ and any $u\in A_i$, there exists an $M$-legged spider of height $j-i$ with center $u$ and leaves in $A_j$.
\item Let $F=\{uv:  u\in A_{j-1}, v\in A_j \mbox{ and } \exists P\in \cF, \, uv\in E(P)\}$.
Then $F$ has minimum degree at least $M$.
\end{enumerate}
\end{lem}
\begin{proof} Fix any $i$ with $0\leq i<j$ and $u\in A_i$. By the definition of $A_i$ there exists $P=v_0v_1\cdots v_j\in \cF$,
where $v_i=u$. Let $I=v_0v_1\cdots v_i$. By condition 5 of Lemma \ref{lem:cleaning1}, $I$ is contained in at least 
$jM(K\delta)^{j-i-1}$ members of $\cF$. In other words, the family $\cQ=\{Q: Q\in A_{i}\times\cdots \times A_j, I\cup Q \in \cF\}$ has size at least $jM(K\delta)^{j-i-1}$. Since each member of $\cQ$ is a path of length $j-i$ fro $u$ to a vertex in $A_j$, 
by Lemma \ref{lem:spider-shrinking2}, there exists a spider of height $j-i$ with center $u$ and leaves in $A_j$ whose number of legs is at least
	\begin{equation*}
	\frac{jM(K\delta)^{j-r-1}}{j(K\delta)^{j-r-1}}=M.
	\end{equation*}
	This proves part 1 of the lemma. Applying part 1 with $i=j-1$, we have $\forall u\in A_{j-1}$, $d_F(u)\ge M$. Let $y\in A_j$. By condition 4 of Lemma \ref{lem:cleaning1}, there exist a vertex $x\in A_0$ such that there are at least $DM$ members of 
$\cF$ that have ends $x,y$. By Lemma \ref{lem:independent-paths}, among these there are at least $DM/[j^2f(j-1,L)^2]\geq M$
of them that are pairwise vertex disjoint outside $\{x,y\}$. In particular, this implies $d_F(y)\ge M$. So part 2 also holds.
\end{proof}

\begin{Definition}
For the rest of the subsection, we write $b=qj+b'$ where $q$ and $b'$ are integers with $1\le b'\le j$. 
\end{Definition}

The next lemma plays an important role in our proof of Lemma \ref{lem:short}. It sets up 
a well-placed $s$-legged spider of height $b$ to be used in building a copy of $t*S^s_{b,k}$.
\begin{lem}\label{lem:b-spider}
Let $G$, $w$, $A_0,\ldots, A_j$, and $\cF$ be as stated in Lemma \ref{lem:cleaning1}.
Let $N=M/D$.
\begin{enumerate}
\item If $q$ is even, then there are $N$-legged spiders $T$ and $T'$ in $G$, both of height $b$ such that the leaf set of $T$ is  contained in $A_b$ and the leaf set of $T'$ is contained in $A_{b-1}$.
\item If $q$ is odd, then there are $N$-legged spiders $T$ and $T'$ in $G$, both of height $b$, such that the leaf set of $T$ is contained in $A_{j-b'}$ and the leaf set of $T'$ is contained in $A_{j-b'+1}$.
\end{enumerate}
\end{lem}
\begin{proof} 
Let $B$ be a bipartite graph with parts $A_0$ and $A_j$ such that $\forall x\in A_0, y\in A_j$ $xy\in E(B)$  if and only if at least $DM$ members of $\cF$ have ends $x,y$.
By Lemma \ref{lem:cleaning1} condition 4, $B$ has minimum degree at least $M$. 
By Lemma \ref{lem:independent-paths}, $\forall xy\in E(B)$, there exist at least $DM/[j^2f(j-1,L)^2]\geq M$ internally disjoint
members of $\cF$ with ends $x,y$.
	
	Fix a vertex $u\in A_0$. Since $\delta(B)\geq M = DN \geq (q+1)N$, where the last inequality follows from Definition \ref{def:constants}, we can greedily grow an $N$-legged spider $R$ in $B$ that has  center $u$ and height $q+1$.
Since $\forall xy\in E(B)$ there are $M$ internally disjoint members of $\cF$ with ends $x,y$,
we can replace each edge $ab$ of $R$ with a member of $\cF$ with ends $a,b$
so that the resulting graph is an $N$-legged spider $S$ of height $(q+1)j$ in $G$. 
Let $A$ be the set of vertices in $S$ that are at distance $b=qj+b'$ from $u$ in $S$.
It is easy to see from the definition of $S$ that if $q$ is even then $A\subseteq A_{b'}$
and that if $q$ is odd then $A\subseteq A_{j-b'}$. Let $T$ be the sub-spider of $S$ with center $u$ and leaf set $A$.
Then $T$ satisfies the first halves of statements 1 and 2.
	
Now, fix a subset $A'_0\subseteq A_0$ of size $N$,
Since $B$ has minimum degree at least $M\geq (q+2)N+1$, in $B$ we can find $N$ disjoint paths of length $q+1$,
$Q_1,\ldots, Q_N$, avoiding $w$, such that $\forall i\in [N]$, $Q_i$ starts from a vertex  $x_i\in A'_0$ .
By a similar reason as in the previous paragraph, we can replace the edges in $\bigcup_{i=1}^N Q_i$ by members of  $\cF$ that avoid $w$ such that for each $i\in [N]$, $Q_i$ is turned into a path $P_i$ of length $(q+1)j$ in $G$ that still avoids $w$ and
that $P_1,\dots, P_N$ are vertex disjoint. Let $S'=\bigcup_{i=1}^N P_i\cup \{wx_1,\ldots, wx_N\}$. Then $S'$ is an
$N$-legged spider in $G$ with center $w$ and height $(q+1)j+1$. Let $A'$ be the set of vertices in $S'$ that are 
distance $b=qj+b'$ from $w$ in $S'$. It is easy to see by the definition of $S'$ that if $q$ is even then $A'\subseteq A_{b'-1}$
and that if $q$ is odd then $A'\subseteq A_{j-b'+1}$. Let $T'$ be the sub-spider of $S'$ with center $w$ and leaf set $A'$.
Then $T'$ satisfies the second halves of statement 1 and 2. 
\end{proof}

\def\tm{\widetilde m}
\begin{lem}\label{lem:additional-spiders}
Let $G$, $w$, $A_0,\ldots, A_j$ and $\cF$ be as stated in Lemma \ref{lem:cleaning1}. 
Let $m,m'$ be positive integers such that $m'\leq \frac{m}{2D}$ and $m\leq \frac{M}{2sk}$.
Let $0\leq r\le j$ and $U \subseteq A_r$ be a subset of size $m$. Let $W$ be a vertex set such that $W\cap U=\emptyset$
and	$|W|\leq M/2$. If $p:=k+r-2j$ is non-negative and even, then there exists an $m'$-legged spider $T'$ with height $k$ and leaf set $U'\subseteq U$ such that $V(T')\setminus U'$ is disjoint from $W\cup U$.
\end{lem}
\begin{proof} Suppose $U=\{u_1,\ldots, u_m\}$.
	Since $U\subseteq A_r$, by Lemma \ref{lem:cleaning2} statement 1, for each $i\in [m]$, there exists an $M$-legged spider of height $j-r$ with center $u_i$ and leaves in $A_j$. Since $M\geq km+|W|\geq (j-r+1)m+|W|$, by a greedy process, 
we can find a collection of vertex disjoint paths $Q_1,\ldots, Q_m$, where for each $i\in [m]$, $Q_i$ is a path
of length $j-r$ joining $u_i$ to a vertex $y_i$ in $A_j$  that avoids the set $W$.
Let \[Y=\{y_1,\ldots, y_m\}\subseteq A_j.\]
 By Lemma \ref{lem:cleaning2} statement 2, the graph 
 \[F=\{ab: a\in A_{j-1}, b\in A_j \mbox{ and } 
 \exists P\in \cF, ab\in E(P)\}\]
  has minimum degree at least $M$. 
 Using a greedy process we can find in $F$  a collection of 
 vertex disjoint paths $R_1,\ldots, R_m$, where for each $i\in [m]$, $R_i$ is a path in $F$ of length $p$ that 
 joins $y_i$ to some vertex $z_i$ and avoids the set $(W\cup \bigcup_{\ell=1}^m V(Q_\ell) )\setminus \{y_i\}$. 
 For each $i\in [m]$, since $y_i\in A_j$ and $p$ is even, $z_i\in A_j$ as well.
 Now $\{Q_i\cup R_i:  i\in [m] \}$ is family of $m$ disjoint paths of length $j-r+p$ that avoids $W$.
 
 Let  $Z=\{z_1,\ldots, z_m\}$. Since $Z\subseteq A_j$, by Lemma \ref{lem:cleaning1} condition 3, 
there exists a set $V_0$ with $A_0\subseteq V_0\subseteq \Gamma_1(w)\setminus A_j$ such that $|V_0|\geq (2K/D)\delta$ and
for each $i\in [m]$, there are at least $|V_0|/D$
 many $x\in V_0$ such that $x,z_i$ are the ends of a $j$-heavy path in $G$. 
 Let $B$ be a bipartite graph with parts $V_0$ and $Z$
 such that $\forall x\in V_0, z\in Z$, $xz\in E(B)$ if and only if $x,z$ are the ends of a $j$-heavy path in $G$. Then
 $e(B)\geq |V_0||Z|/D$.  Since $|Z|/D=m/D\geq 2m'$. So, by Lemma \ref{lem:common-degree}
there exists an $m'$-set $Z'\subseteq Z$ such that 
\[|N^*_B(Z')|\geq \left(\frac{1}{2D}\right)^{m'} |V_0|\geq \left(\frac{1}{2D}\right)^M |V_0|.\]

Since $|V_0|\geq (K/2D)\delta$, while $\delta=\omega(1)$, 
for sufficiently large $n$, we have $|N^*_B(Z')|> |\bigcup_{i=1}^m V(Q_i\cup R_i) \cup W|$.
So, there exists a vertex $x\in V_0\setminus (\bigcup_{i=1}^m V(Q_i\cup R_i) \cup W)$
that is joined to all of $Z'$ in $B$.  Without loss of generality, suppose $Z'=\{z_1,\ldots, z_{m'}\}$.  For each $i\in [m']$,
there exists a $j$-heavy path with ends $x$ and $z_i$. In particular, by Lemma \ref{lem:independent-paths}, there are at least $\frac{f(j,L)}{j^2f(j-1,L)^2}\ge\frac{10j^4D^2 M}{j^2D}\ge M$ internally disjoint paths of length $j$ between $x$ and $z_i$, where the first inequality follows by Definition \ref{def:constants} and \eqref{eq:f-facts}. As $M>m'j+|W|$, we can find paths $P_1,\ldots, P_{m'}$, where $\forall i\in [m']$, $P_i$ is
a path of length $j$ joining $x$ to $z_i$ such that $T':=\bigcup_{i=1}^{m'} (P_i\cup Q_i\cup R_i)$ is an $m'$-legged spider
of height $k$ with center $x$ and leaf set $\{u_1,\dots, u_{m'}\}$ and such that $T'$ avoids  $(U\setminus \{u_1,\dots, u_{m'}\})\cup W$.
The lemma holds for the above-defined $T'$ and $U'=\{u_1,\ldots, u_{m'}\}$.
\end{proof}

Now, we are ready to prove Lemma \ref{lem:short}. 

\medskip

\noindent{\bf Proof of Lemma \ref{lem:short}:}
Suppose that the number of $j$-heavy paths in $G$ is at least $\frac{(j+1)^{j+1}}{L}n\delta^j$. 
		Let $w$, $A_0,\ldots, A_j$ and $\cF$ are obtained by Lemma \ref{lem:cleaning1}. 
Our first step is to find an appropriate value of $r$ to apply Lemma \ref{lem:additional-spiders} to.
 Recall that $b=qj+b'$. Let 
		\begin{equation*}
		r:=\left\{
		\begin{array}{lr}
		b', &\mbox{ if $q$ is even and $k+b'-2j$ is even}  \\
		b'-1, & \mbox{ if $q$ is even and $k+b'-2j$ is odd}\\
		j-b', & \mbox{ if $q$ is odd and $k+(j-b')-2j$ is even} \\
		j-b'+1,& \mbox{ if $q$ is odd and $k+(j-b')-2j$ is odd}
		\end{array}
		\right.
		\end{equation*}
As $1\le b'\le j$, we have $0\le r\le j$. 

Let $p=k+r-2j$. By the definitions of $p$ and $r$, it is easy to see $p$ is even. 
We claim that $p$ is non-negative. To prove this, it is enough to show that $k+b'-2j\ge0$ when $q$ is even, and that $k+(j-b')-2j\ge0$ when $q$ is odd. First assume that $q$ is even. If $q=0$, then $b=b'$ and $k+b'-2j=k+b-2j\ge0$ where the inequality holds by our assumption $j\le \frac{k+b}{2}$; if $q\ge2$, then $b\ge 2j+b'$ and thus $k+b'-2j\ge b+b'-2j\ge 2b'>0$. Now assume $q$ is odd. Then we have that $q\ge1$ and thus $b\ge j+b'$. It follows that $k+(j-b')-2j=k-(j+b')\ge b-(j+b')\ge0$. Hence $p$ is nonnegative.

Now, let $m_1=M/D$ and for $i=2,\ldots, s$, let $m_i=m_{i-1}/(2D)$. Using the definition of $D$, it is easy
to check that $\forall i\in [s], m_i\leq M/(2sk)$.
By Lemma \ref{lem:b-spider}, there exists an $m_1$-legged spider $T_1$ with height $b$ and leaf set $U_1\subseteq A_{r}$. 
The idea of the rest of the proof is to apply Lemma \ref{lem:additional-spiders} $s-1$ times. Initially let $W=V(T_1)\setminus U_1$.
Since $p=k+r-2j$ is non-negative and even, applying Lemma \ref{lem:additional-spiders} with $m_1$ and $m_2$ playing the roles of
$m$ and $m'$ respectively and  $U_1$ playing the role of $U$, we can find an $m_2$-legged spider $T_2$ with height $k$ and leaf set $U_2\subseteq U_1$ such that $V(T_2)\setminus U_2$ is disjoint from $W\cup U_1$. 
Now, we add $V(T_2) \setminus U_2$ to $W$. Next, 
applying Lemma \ref{lem:additional-spiders} with $m_2,m_3$ playing the roles of $m$ and $m'$ respectively and $U_2$ playing the role of $U$,  we can find an $m_3$-legged spider $T_3$ with height $k$ and leaf set $U_3\subseteq U_2$, such that $V(T_3)\setminus U_3$ is disjoint from $W\cup U_2$. Now, we add $V(T_3)\setminus U_3$ to $W$. We continue like this.
It is easy to check that we can carry out the process for at least $s-1$ steps to find $T_2,\ldots, T_s$. Indeed, within
the first $s-1$ steps $W$ has size at most $km_1+km_2+\cdots+km_{s-1}<2km_1=2kM/D<M/2$.
This together with the definitions of $m_1,\ldots, m_s$ ensures that the conditions of Lemma \ref{lem:additional-spiders} are satisfied. But now $\bigcup_{i=1}^s T_i$ forms a copy of $t*S^s_{b,k}$ in $G$, a contradiction. This completes our proof of Lemma \ref{lem:short}.
\qed

\subsubsection{Proof of Lemma \ref{lem:heavy-paths}}

Now we are in a position to prove Lemma \ref{lem:heavy-paths}.

\noindent\textbf{Proof of Lemma \ref{lem:heavy-paths}:} By Lemmas \ref{lem:long} and \ref{lem:short}, for any $2\le j\le k$, the number of $j$-heavy paths in $G$ is at most $\max\{\frac{n\delta^j}{L},\frac{(j+1)^{j+1}}{L}n\delta^j\}=\frac{(j+1)^{j+1}}{L}n\delta^j$. This completes the proof.\qed


\section{Concluding remarks} \label{sec:more-exponents}
In \cite{KKL}, Kang, Kim and Liu extended the definition of balanced rooted trees to that of a balanced rooted bipartite graphs as follows. Let $F$ be a bipartite graph and $R$ a proper subset of $V(F)$ called the set of {\it roots}. For each nonempty
set $S\subseteq V(F)$, let $\rho_F(S)=\frac{e_S}{|S|}$, where $e_S$ is the number of edges in $G$ with at least one end in $S$.
Let $\rho(F)=\rho_F(V(F)\setminus R)$. We say that $(F,R)$ is {\it balanced} if $\rho_F(S)\geq \rho(F)$ for every nonempty subset $S\subseteq V(F)\setminus R$.  A real number $r\in(1, 2)$ is called \emph{balancedly realizable} if there is a connected bipartite
graph $F$ and a set $R\subseteq V(F)$ such that $(F,R)$ is balanced with $\rho_F= \frac{1}{2-r}$ and that
there is a positive integer $t_0$ such that for all integers $t\ge t_0$, $\ex(n,t*F)=\Theta(n^r)$ holds. By definition, a balancedly realizable number is a Tur\'an exponent. Using a result of Erd\H{o}s and Simonovits \cite{cube}, Kang, Kim and Liu \cite{KKL} proved the following.

\begin{lem}[\cite{KKL}] \label{lem:KKL}
	Let $a<b$ be two integers. If $2-\frac{a}{b}$ is balancedly realizable, then $2-\frac{a}{a+b}$ is also balancedly realizable.
\end{lem}

\noindent{\bf Proof of Corollary \ref{cor:dense1} and Corollary \ref{cor:dense2}:}
By Theorem \ref{BC-theorem} and Theorem \ref{thm:subdivision-bound}, for any positive integers $p,k,b$ with $k\ge b$, $1+\frac{p}{kp+b}$ is balancedly realizable. Corollay \ref{cor:dense1} follows by applying Lemma \ref{lem:KKL} repeatedly. 

Now, suppose that $q=sp+p'$ where $s$ is an positive integer and $0\le p'\le\sqrt{p}$. Since it is known that $2-\frac{1}{s}$ is a Tur\'an exponent for all any integer $s\ge 2$, we may assume $p'>0$. Now, as $0<p'\le\sqrt{p}$, there exists integers $k$ and $b$ such that $p=kp'+b$ and $k\ge p'-1$ and $p'\ge b\ge1$. Then $2-\frac{p}{q}=2-\frac{kp'+b}{s(kp'+b)+p'}$. By Corollary \ref{cor:dense1}, it follows that $2-\frac{p}{q}$ is a Tur\'an exponent.
\qed

\medskip

Finally, even though we obtained all the Tur\'an exponents that Janzer's conjecture (Conjecture \ref{conj:Janzer}) would give,
it would still be very interesting to resolve his conjecture in the full. While the rational exponent conjecture is a central problem in the study of bipartite Tur\'an problems, the ultimate goal is to understand the Tur\'an function for bipartite graphs better. In particular,
while tools such as dependent random choice have found success in the denser end of the spectrum for bipartite graphs,
it would be very interesting to develop more tools for the sparser end of the spectrum. The recent active study of the Tur\'an problem for subdivisions is a step in that direction. It will be very interesting to continue explore problems of such nature.

\appendix 

\section{Proof of Lemma \ref{lem:heavy-spiders}}

As mentioned in the paper, the proof of Lemma \ref{lem:heavy-spiders} follows from similar arguments used in the main proof of \cite{Janzer2}. We give a sketch of the proof to make our paper self-contained.
We split the proof of Lemma \ref{lem:heavy-spiders} into two lemmas: Lemma \ref{lem : j1=j2=1} and Lemma \ref{lem:bk-spiders-atmost1}.

\medskip
\begin{lem}\label{lem : j1=j2=1}
Let $G$ be a $t*S^s_{b,k}$-free $K$-almost-regular graph on $n$ vertices with minimum degree $\delta=\omega(1)$. Let $1\le j_1\le b$ and $1\le j_2,\ldots,j_s\le k$ be integers and suppose that $j_i=1$ holds for at least two values of $i$. Then the number of heavy spiders with
	length vector $(j_1,\ldots,j_s)$ is at most $\frac{27K^{j-2}}{L}n\delta^j$, where $j=j_1+\cdots+j_s$.
\end{lem}
\begin{proof}
	If $s = 2$, then the result follows by Lemma \ref{lem:heavy-paths}, since a spider with length vector
	$(1, 1)$ is heavy if and only if it is heavy when viewed as a path of length 2. So we may assume that
	$s\ge3$. In this case we have that $j\ge3$. Let $\cS$ denote the family of all heavy spiders 
in $G$ with length vector $(j_1,\ldots,j_s)$. Assume that $j_{i_1} = j_{i_2} = 1$. 
	By definition, there are at least $f(j,L)$ admissible spiders with leaf vector $(x_1,\ldots,x_s)$ and length vector $(j_1,\ldots,j_s)$. By Lemma \ref{lem:independent-paths} and Proposition \ref{prop}, among these spiders there are 
	$$\frac{{f(j,L)}}{j^2f(j-1,L)^2}\ge f(j-1,L)$$ 
	internally disjoint spiders, which give us at least $f(j-1)\ge f(2,L)$ common neighbors of $x_1$ and $x_2$. Thus
there are at least $f(2,L)$ paths of length $2$ with ends $x_{i_1},x_{i_2}$. By Definition \ref{def:admissible-paths}, any of these paths
is $2$-heavy. In particular, $x_{i_1}wx_{i_2}$ is $2$-heavy. Since $G$ has maximum degree at most $K\delta$, at most
$(K\delta)^{j-2}$ different members $P$ of $\cS$ can give rise to the same $x_{i_1}wx_{i_2}$.
By Lemma \ref{lem:heavy-paths}, the number of 2-heavy paths is at most $\frac{27n\delta^{2}}{L}$. 
Hence $|\cS|\leq \frac{27n\delta^2}{L}\cdot (K\delta)^{j-2}
=\frac{27K^{j-2}}{L}n\delta^j$.
\end{proof}

\begin{lem}\label{lem:bk-spiders-atmost1}
	Let $G$ be a $K$-almost-regular graph on $n$ vertices with minimum degree $\delta=\omega(1)$. Let $t,j_1,\ldots,j_s,k_1,\ldots,k_s$ be positive integers with each $j_i\le k_i$. Suppose that $j_i=1$ holds for at most one value of $i$. Then provided that $L$ is sufficiently large, if the number of heavy spiders with length vector $(j_1,\ldots,j_s)$ is at least $\frac{n\delta^{j}}{L}$ where $j=j_1+\cdots+j_s$, there exist $t$ internally disjoint spiders in $G$ with the same leaf vector and length vector $(k_1,\ldots,k_s)$.
\end{lem}

\begin{Definition}
	Let $\cF$ is a family of spiders in $G$. 
	\begin{equation*}
	\partial (\cF)=\{T: T\mbox{ is a proper subtree of some } F\in\cF \}.
	\end{equation*}
	For  each $T\in\partial(\cF)$, we define $\cF|_{T}$ to be the subfamily of members of $\cF$ that contain $T$.
\end{Definition}

\begin{lem}\label{lem: cleanning spiders}
	Let $G$ be a $K$-almost-regular graph on $n$ vertices with minimum degree $\delta=\omega(1)$. Suppose that the number of heavy spiders with length vector $(j_1,\ldots,j_s)$ is at least $\frac{n\delta^{j}}{L}$  where $j=j_1+\cdots+j_s$. Then provided that $L$ is sufficiently large, there exists a non-empty family $\cF$ of admissible spiders with length vector $(j_1,\ldots,j_s)$ such that the following hold.
\begin{enumerate}
\item For  each $S\in \cF$, at least $f(j,L)/2$ member of $\cF$ share the same leaf vector as $S$.
\item For any $T\in\partial(\cS)$, $|\cF|_{T}|\ge (K\delta)^{j-e(T)}/L^2$.
\end{enumerate}
\end{lem}
\begin{proof} Let $\cF^*$ be the family of all heavy spiders in $G$ with length vector $(j_1,\ldots,j_s)$. Suppose that $|\cF^*|\ge\frac{n\delta^j}{L}$.
For each vector $(x_1,\ldots, x_s)$ of $s$ distinct vertices in $G$, let $\cF^*_{(x_1,\ldots, x_s)}$ denote the
subfamily of members of $\cF$ that have leaf vector $(x_1,\ldots, x_s)$. By the definition of $\cF^*$,
for each $(x_1,\ldots, x_s)$ in $G$, $|\cF^*_{(x_1,\ldots, x_s)}|$ is either $0$ or at least $f(j,L)$.
Let $X$ denote the set of those $(x_1,\ldots, x_s)$ for which $|\cF^*_{(x_1,\ldots, x_s)}|\geq f(j,L)$.
Then 
\begin{equation}\label{X-bound}
|X|\leq |\cF^*|/f(j,L).
\end{equation}

Initially, let $\cF=\cF^*$ and for each $(x_1,\ldots, x_s)$ let $\cF_{(x_1,\ldots, x_s)}$ be the subfamily
of members of $\cF$ that have leaf vector $(x_1,\ldots, x_s)$.
We now do the following two types of cleaning on $\cF$. We update $\cF$ immediately after each step.
Type 1: if there exists some $T\in \partial(\cF)$ that is contained
in fewer than $(K\delta)^{j-e(T)}/L^2$ members of $\cF$ remove all the member of $\cF$ containing $T$.
Type 2: if there exists a vector $(x_1,\ldots, x_s)$ of $s$ distinct vertices such that $0<|\cF_{(x_1,\ldots, x_s)}|< \frac{f(j,L)}{2}$,
we remove all the members in $\cF_{(x_1,\ldots, x_s)}$ from $\cF$. We continue until either $\cF$ becomes empty or no more removal can be performed. It suffices to show that the final $\cF$ is non-empty as it clearly satisfies the requirements of the lemma.

To that end, note that the total number of members removed by a type 2 removal is fewer than $|X|\cdot \frac{f(j,L)}{2}\leq |\cF^*|/2$.
Now, we bound the number of members removed by a type 1 removal.
By Cayley's formula, the number of trees on $i$ vertices is at most $i^{i-2}$. Since $G$ has maximum degree at most $K\delta$, the number of deleted because of some $T\in\partial(\cF)$ being contained in fewer than $(K\delta)^{j-e(T)}/L^2$ members of $\cF$ is no more than
	\begin{equation*}
	\sum_{i=2}^{b+(s-1)k-1}\binom{sk}{i-1}i^{i-2} n(K\delta)^{i-1}\cdot(K\delta)^{j-(i-1)}/L^2\le\frac{(sk)^{2sk}n(K\delta)^j}{L^2}=\frac{n\delta^j}{2L}\cdot\frac{2(sk)^{2sk}K^j}{L},
	\end{equation*}
	which is less than $\frac{n\delta^{j}}{2L}\le\frac{|\cF^*|}{2}$ when $L$ is a sufficiently large constant; here the factor $\binom{sk}{i-1}$ upper bounds the number of positions of an $i$-vertex tree in a spider with length vector $(j_1,\ldots,j_s)$. So altogether we have removed fewer than $|\cF^*|$ members from $\cF^*$. So the final $\cF$ is non-empty. This completes our proof.
\end{proof}

\begin{lem}\label{lem: (b,k,...,k) spiders}
	Let $G$ be a $K$-almost-regular graph on $n$ vertices with minimum degree $\delta=\omega(1)$. Let $j_1,\ldots,j_s$ and $t$ be positive integers and assume that $j_i=1$ holds for at most one value of $i$. Suppose that $\cF$ is a non-empty family of admissible spiders with length vector $(j_1,\ldots,j_s)$ satisfying the conditions in Lemma \ref{lem: cleanning spiders}. Then for any integers $k_1,\ldots,k_s$ with each $k_i\ge j_i$, provided that $L$ is sufficiently large, there exists an $s$-tuple $(v_1,\ldots,v_s)$ of distinct vertices such that the following holds. For any vertex set $Z$ of size at most $L$ that is disjoint from $\{v_1,\ldots,v_s\}$, there exist an $s$-legged spider with leaf vector $(v_1,\ldots,v_s)$ and length vector $(k_1,\ldots,k_s)$ that is disjoint from $Z$. 
	
	In particular, there are $t$ internally disjoint spiders in $G$ with leaf vector $(v_1,\ldots,v_s)$ and length vector $(k_1,\ldots,k_s)$.
\end{lem}
\begin{proof} 
	For each $i\in[s]$, choose $\gamma_i\in\{0,1\}$ such that $k_i-j_i-\gamma_i$ is even. Let $k=\max\{k_1,\ldots,k_s \}$. Since $k_i-j_i-\gamma_i$ is an even integer between $0$ and $k$, there exist  $\eta_{i,1},\ldots,\eta_{i,k}\in\{0,1\}$ such that $k_i-j_i-\gamma_i=2\eta_{i,1}+\cdots+2\eta_{i,k}$. Let $L$ be a sufficiently large constant.
	
	Let $R_0$ be a subspider of some $S\in\cF$ with length vector $(j_1-\gamma_1,\ldots,j_s-\gamma_s)$. Here and throughout the proof, we allow that a subspider has legs of length $0$, and in this case, the leaf on its leg of length 0 is defined to be the center of this spider. Let $(v_1,\ldots,v_s)$ be the leaf vector of $R_0$. Since $j_i=1$ holds for at most one value of $i$, there is at most one leg of $R_0$ having length $0$. Therefore $v_1,\ldots,v_s$ are distinct vertices.
	
	By Condition 1 in Lemma \ref{lem: cleanning spiders}, and by Lemma \ref{lem:independent-spiders} and Proposition \ref{prop}, for each $S\in\cF$ we can fix a family $\cT(S)\subseteq\cF$ of $L^2$ internally disjoint spiders with the same leaf vector as $S$.
	Next we will define spiders $R_1,\ldots,R_k$, $S_1,\ldots,S_{k+1},T_1,\ldots,T_{k+1}$, with which we can build a desired spider.

	Since $R_0\in\partial(\cF)$, by Condition 2 in Lemma \ref{lem: cleanning spiders}, the number of spiders in $\cF$ that contain $R_0$ is  $|\cF|_{R_0}|\ge(K\delta)^{j-e(R_0)}/L^2=(K\delta)^{\gamma_1+\cdots+\gamma_s}/L^2$. As the maximum degree of $G$ is at most $K\delta$, the number of spiders that contain $R_0$ and some vertex in $Z\setminus V(R_0)$ is at most $|Z|(K\delta)^{j_1+\cdots+j_s -e(R_0)-1}=O(\delta^{\gamma_1+\cdots+\gamma_s-1})<|\cF|_{R_0}|$, where the last inequality holds because of $\delta=\omega(1)$. Since the leaf set of $R_0$ is disjoint from $Z$, it follows that there exists a spider $S_1'\in\cF|_{R_0}$ whose leaf set is disjoint from $Z$. As $\cT(S_1')$ is a family of internally disjoint spiders of size $L^2>|Z|+2$, there exist members $S_1,T_1$ of $\cT(S_1')$ such that $S_1$ and $T_1$ are disjoint from $Z$. Let $R_1$ be the subspider of $T_1$ with length vector $(j_1-\eta_{1,1},\ldots,j_s-\eta_{s,1})$.
	
	Iteratively, for $1\le \ell\le k$, suppose we have defined $R_{\ell}$ of length vector $(j_1-\eta_{1,\ell},\ldots,j_s-\eta_{s,\ell})$ which is a subspider of $T_{\ell}\in\cF$. We define $ S_{\ell+1},T_{\ell+1}$ and $R_{\ell+1}$ as follows.
	
	Choose one $S_{\ell+1}\in\cF|_{R_{\ell}}$ such that $V(S_{\ell+1})-V(R_{\ell})$ is disjoint from $Z\cup (V(S_1)\cup\cdots\cup V(S_{\ell}))\cup(V(T_1)\cup\cdots\cup V(T_{\ell}))$. This is possible by Lemma \ref{lem: cleanning spiders}(ii).	
	Then, choose one $T_{\ell+1}\in\cT(S_{\ell+1})\setminus\{S_{\ell+1}\}$ such that $V(T_{\ell+1})$ is disjoint from $Z\cup (V(S_1)\cup\cdots\cup V(S_{\ell}))\cup(V(T_1)\cup\cdots\cup V(T_{\ell}))$. This is possible  as $|\cT(S_{\ell+1})|=L^2>|Z|+(sk+1)k$.	
	Finally, if $\ell< k$, let $R_{\ell+1}$ be the subspider of $T_{\ell+1}$ with length vector $(j_1-\eta_{1,{\ell+1}},\ldots,j_s-\eta_{s,{\ell+1}})$.
	
	Now for $\ell\ge 1$, let $S_\ell$ have leaf vector $(x_{\ell,1},\ldots,x_{\ell,s})$ and let $R_\ell$ have leaf vector $(r_{\ell,1},\ldots,r_{\ell,s})$. Then for each $i\in[s]$, $v_i x_{1,i}$ forms a path of length $\gamma_i$ and for each $\ell\in[k]$ $x_{\ell,i}r_{\ell,i}x_{\ell+1,i}$ forms a path of length $2\eta_{i,\ell}$. By our definitions, for each $i\in[s]$ the vertex sequence $v_i x_{1,i}r_{1,i}x_{2,i}\cdots x_{k,i} r_{k,i}x_{k+1,i}$ forms a path $P_i$ of length $\gamma_i+2\eta_{i,1}+\cdots+2\eta_{i,k}=k_i-j_i$ that avoids $Z$. Since $S_\ell$ has length vector $(j_1,\ldots,j_s)$ and each $j_i>0$, its leaves $x_{\ell,1},\ldots,x_{\ell,s}$ are distinct. Since $R_\ell$ has length vector $(j_1-\eta_{1,\ell},\ldots,j_s-\eta_{s,\ell})$ and $j_i=1$ holds for at most one value of $i$, there is at most one leg of $R_\ell$ having length 0. Thus the leaves $r_{\ell,1},\ldots,r_{\ell,s}$ of $R_\ell$ are distinct. Now we can conclude that $P_1,\ldots,P_s$ are vertex disjoint paths that are disjoint from $Z$, and that each $P_i$ has length $k_i-j_i$ and has ends $v_i$ and $x_{k+1,i}$.
	
	By our choice of $T_{k+1}$, $T_{k+1}$ is a spider with leaf vector $(x_{k+1,1},\ldots,x_{k+1,s})$ and length vector $(j_1,\ldots,j_s)$, and $T_{k+1}$ is disjoint from $Z$ and intersects $\cup_{i=1}^s P_i$ only on its leaves. So $T:=T_{k+1}\cup (\cup_{i=1}^s P_i)$ is a spider with leaf vector $(v_1,\ldots,v_s)$ and length vector $(k_1,\ldots,k_s)$, and $T$ is disjoint from $Z$.
	
	Next we prove the particular part. Let $Z_1=\emptyset$. Then we can find a spider $T_1$ with leaf vector $(v_1,\ldots,v_s)$ and length vector $(k_1,\ldots,k_s)$. Iteratively, for $2\le i\le t$, let $Z_i=\cup_{\ell=1}^{i-1}(V(T_i)-\{v_1,\ldots,v_s\})$. Then $|Z_i|\le skt\le L$. So we can find a spider $T_i$ with leaf vector $(v_1,\ldots,v_s)$ and length vector $(k_1,\ldots,k_s)$ that is disjoint from $Z_i$. This allows us to find $t$ internally disjoint spiders $T_1,\ldots,T_t$ with leaf vector $(v_1,\ldots,v_s)$ and length vector $(k_1,\ldots,k_s)$. The proof is completed.
\end{proof}

\noindent\textbf{Proof of Lemma \ref{lem:bk-spiders-atmost1}:}
	Suppose that the number of heavy spiders in $G$ with length vector $(j_1,\ldots,j_s)$ is at least $\frac{n\delta^{j}}{L}$. By Lemma \ref{lem: cleanning spiders}, there exists a family of spiders  with length vector $(j_1,\ldots,j_s)$ satisfying the conditions in Lemma \ref{lem: cleanning spiders}. Then the result follows by Lemma \ref{lem: (b,k,...,k) spiders}. 	
\qed

\medskip

\noindent
\textbf{Proof of Lemma \ref{lem:heavy-spiders}:}
We may assume that $j_i=1$ holds for at most one value of $i$, as otherwise the lemma follows easily by Lemma \ref{lem : j1=j2=1}. Now suppose for a contrary that the number of heavy spiders in $G$ with length vector $(j_1,\ldots,j_s)$ is at least $\frac{K^{j-2}}{L}n\delta^{j}$, where $j=j_1+\cdots+j_s$. As $\frac{27K^{j-2}}{L}n\delta^{j}\ge \frac{n\delta^j}{L}$, applying Lemma \ref{lem:bk-spiders-atmost1} with $(k_1,\ldots,k_s)=(b,k,\ldots,k)$, we can find a copy of $t*S^s_{b,k}$ in $G$, which contradicts $G$ being $t*S^s_{b,k}$-free. 
\qed

\end{document}